\documentclass[12pt, reqno]{amsart}
\usepackage{amsmath, amsthm, amscd, amsfonts, amssymb, graphicx, color}
\usepackage[bookmarksnumbered, colorlinks, plainpages]{hyperref}
\hypersetup{colorlinks=true,linkcolor=red, anchorcolor=green, citecolor=cyan, urlcolor=red, filecolor=magenta, pdftoolbar=true}

\newtheorem{theorem}{Theorem}[section]

\newtheorem{proposition}[theorem]{Proposition}
\newtheorem{corollary}[theorem]{Corollary}
\theoremstyle{definition}

\newtheorem{example}[theorem]{Example}

\theoremstyle{remark}
\newtheorem{remark}[theorem]{Remark}
\numberwithin{equation}{section}

\newcommand{\be}{\begin{equation}}
\newcommand{\ee}{\end{equation}}

\newcommand{\cM}{{\mathcal M}}
\newcommand{\CC}{\mathbb{C}} 
\newcommand{\NN}{\mathbb{N}}
\newcommand{\RR}{\mathbb{R}}

\newcommand{\al}{\alpha}

\newcommand{\AAA}{\mathcal{A}}

\begin{document}
\setcounter{page}{1}

\title[Inequalities on the spectral radius of Hadamard products]{Inequalities on the spectral radius and the operator norm of Hadamard products of positive operators on sequence spaces}

\author[R. Drnov\v{s}ek and A. Peperko]{Roman Drnov\v{s}ek$^1$ and Aljo\v{s}a Peperko$^2$}

\address{$^{1}$ Faculty of Mathematics and Physics, University of Ljubljana, Jadranska 19, SI-1000 Ljubljana, Slovenia.}
\email{\textcolor[rgb]{0.00,0.00,0.84}{roman.drnovsek@fmf.uni-lj.si}}

\address{$^{2}$  Faculty of Mechanical Engineering, University of Ljubljana, A\v{s}ker\v{c}eva 6, SI-1000 Ljubljana, Slovenia;
\newline
Institute of Mathematics, Physics, and Mechanics,
Jadranska 19, SI-1000 Ljubljana, Slovenia.}
\email{\textcolor[rgb]{0.00,0.00,0.84}{aljosa.peperko@fmf.uni-lj.si; aljosa.peperko@fs.uni-lj.si}}



\begin{abstract}
Relatively recently, K.M.R. Audenaert (2010), R.A. Horn and F. Zhang (2010), Z. Huang (2011), A.R. Schep (2011), A. Peperko (2012), D. Chen and Y. Zhang (2015) have proved inequalities on the spectral radius and the operator norm
 of Hadamard products and ordinary matrix products 
of finite and infinite non-negative matrices that define operators on sequence spaces. 
In the current paper we extend and refine several of these results and also prove some analogues for the numerical radius.
Some inequalities seem to be new even in the case of $n\times n$ non-negative matrices.
\end{abstract} \maketitle

\section{Introduction}

\baselineskip 6mm

In \cite{Zh09}, X. Zhan conjectured that, for non-negative $n\times n$ matrices $A$ and $B$, the spectral radius $\rho (A\circ B)$ of the Hadamard product satisfies
$$\rho (A\circ B) \le \rho (AB),$$
where $AB$ denotes teh usual matrix product of $A$ and $B$. This conjecture was confirmed by K.M.R. Audenaert in \cite{Au10} by proving 
\be
\rho (A\circ B) \le \rho ^{\frac{1}{2}}((A\circ A)(B\circ B))\le \rho (AB).
\label{Aud}
\ee
These inequalities were established via a trace description of the spectral radius. Using the fact that the Hadamard product is a principal submatrix of 
the Kronecker product, R.A. Horn and F. Zhang  proved in \cite{HZ10} the inequalities
\be
\rho (A\circ B) \le \rho ^{\frac{1}{2}}(AB\circ BA)\le \rho (AB)
\label{HZ}
\ee
and also the right-hand side inequality in (\ref{Aud}). Applying the techniques of \cite{HZ10}, Z. Huang proved that 
\be
\rho (A_1 \circ A_2 \circ \cdots \circ A_m) \le \rho (A_1 A_2 \cdots A_m)
\label{Hu}
\ee
for $n\times n$ non-negative matrices $A_1, A_2, \cdots, A_m$ (see \cite{Hu11}). A related inequality 
for $n\times n$ non-negative matrices was shown in \cite{EJS88}: 
\be
\rho (A_1 \circ A_2 \circ \cdots \circ A_m) \le \rho (A_1) \rho (A_2) \cdots \rho (A_m ) .
\label{EJS}
\ee

In \cite{S11} and \cite{Sc11}, A.R. Schep extended inequalities (\ref{Aud}) and (\ref{HZ}) to non-negative matrices that define bounded 
operators on sequence spaces (in particular on $l^p$ spaces, $1\le p <\infty$). In the proofs certain results on the Hadamard product from \cite{DP05} were used.  
It was claimed in \cite[Theorem 2.7]{S11} that
\be
\rho (A\circ B) \le \rho ^{\frac{1}{2}}((A\circ A)(B\circ B))\le  \rho ^{\frac{1}{2}}(AB\circ BA)\le \rho (AB).
\label{Sfirst}
\ee
However, the proof of \cite[Theorem 2.7]{S11} actually demonstrates that
\be
\rho (A\circ B) \le \rho ^{\frac{1}{2}}((A\circ A)(B\circ B))\le  \rho ^{\frac{1}{2}}(AB\circ AB)\le \rho (AB).
\label{Sproved}
\ee
It turned out that $\rho (AB\circ BA)$ and $\rho (AB\circ AB)$ may in fact be different and that $(\ref{Sfirst})$ is false in general. This typing error was corrected in
\cite{Sc11} and \cite{P12}. Moreover, it was proved in \cite{P12} that  for non-negative matrices that define bounded 
operators on sequence spaces 
the inequalities
\be
\rho (A\circ B) \le \rho ^{\frac{1}{2}}((A\circ A)(B\circ B))\le \rho(AB \circ AB)^{\frac{1}{4}} \rho(BA \circ BA)^{\frac{1}{4}} \le \rho (AB)
\label{mix1}
\ee
and (\ref{Hu}) hold.
 
 In \cite{CZ15}, by applying the techniques of \cite{Au10} the inequality (\ref{Hu}) in the case of $n\times n$ non-negative matrices was interpolated in the sense 
$$\rho (A_1 \circ A_2 \circ \cdots \circ A_m) \le [ \rho (A_1 \circ A_2 \circ \cdots \circ A_m)]^{1- \frac{2}{m}}[\rho ((A_1 \circ A_1)(A_2 \circ A_2) \cdots (A_m \circ A_m))] ^{\frac{1}{m}} $$
\begin{equation}
\le  \rho (A_1 A_2 \cdots A_m)
\label{CZ}
\end{equation}
for $m\ge 2$.

The paper is organized as follows. In the second section we introduce some definitions and facts and recall some results from \cite{DP05} and \cite{P06}, which we will need in our proofs. In the third section we extend and/or refine several inequalities from \cite{Hu11}, \cite{P12}, \cite{CZ15},  \cite{DP05} and \cite{P06} (including the inequalities (\ref{Hu}) and (\ref{CZ})) to   non-negative matrices that define bounded 
operators on sequence spaces. 
 More precisely, in Theorem \ref{Bfs} we prove a version of inequality (\ref{Hu}), which is valid for arbitrary positive kernel operators on Banach function spaces. In Theorem \ref{genP1} we refine inequality (\ref{Hu}) and prove  analogues for the operator norm and the numerical radius. Consequently, Corollary \ref{fromCZ} generalizes and refines (\ref{CZ}).  In Theorem \ref{decreasing} we refine the inequality (\ref{EJS}) and prove analogue results for  the operator norm and the numerical radius. We generalize and refine some additional results from  \cite{Hu11} and \cite{CZ15} in Theorems \ref{th4Hu} and \ref{th5H}.
We conclude the paper by applying the spectral mapping theorem to obtain additional results (Theorem \ref{powerineq}, Corollaries \ref{exp} and \ref{resolv}).
Several inequalities in the paper appear to be  new even in the case of $n\times n$ non-negative matrices.

\section{Preliminaries}
\vspace{1mm}

Let $R$ denote either the set $\{1, \ldots, n\}$ for some $n \in \NN$ or the set $\NN$ of all natural numbers.
Let $S(R)$ be the vector lattice of all complex sequences $(x_i)_{i\in R}$.
A Banach space $L \subseteq S(R)$ is called a {\it Banach sequence space} if $x \in S(R)$, $y \in L$
and $|x| \le |y|$ imply that $x \in L$ and $\|x\|_L \le \|y\|_L$. 
The cone of all non-negative elements in $L$ is denoted by $L_{+}$.

Let us denote by $\mathcal{L}$ the collection of all Banach sequence spaces
$L$ satisfying the property that $e_i = \chi_{\{i\}} \in L$ and
$\|e_i\|_L=1$ for all $i \in R$. Standard examples of spaces from $\mathcal{L}$ are Euclidean spaces, the 
well-known spaces $l^p (R)$ ($1\le p \le \infty$) and the space $c_0$ 
of all null convergent sequences, equipped with the usual norms. The set $\mathcal{L}$ also contains all cartesian products $L=X\times Y$ for 
$X, Y\in \mathcal{L}$, equipped with the norm
$\|(x, y)\|_L=\max \{\|x\|_X, \|y\|_Y\}$.

A matrix $A=[a_{ij}]_{i,j\in R}$ is called {\it non-negative} if $a_{ij}\ge 0$ for all $i, j \in R$. 
Given matrices $A$ and $B$, we write $A \le B$ if the matrix $B - A$ is non-negative. Note that the matrices here need not be finite dimensional. 

By an {\it operator} on a Banach sequence space $L$ we always
mean a linear operator on $L$. We say that a non-negative matrix $A$ defines an operator on $L$ if $Ax \in L$ for all $x\in L$, where
$(Ax)_i = \sum _{j \in R}a_{ij}x_j$. Then $Ax \in L_+$ for all $x\in L_+$ and
so $A$ defines a {\it positive} operator on $L$. Recall that this operator is always bounded, i.e., its operator norm
\be
\|A\|=\sup\{\|Ax\|_L : x\in L, \|x\|_L \le 1\}=\sup\{\|Ax\|_L : x\in L_+, \|x\|_L \le 1\}
\label{equiv_op}
\ee
is finite.  
Also, its spectral radius $\rho (A)$ is always contained in the spectrum.
We will frequently use the equality $\rho (S T) = \rho (T S)$ that holds for all bounded operators $S$ and $T$ on a Banach space.

If $A=[a_{ij}]$ is a non-negative matrix that define an operator on $l^2(R)$, then 
the matrix $A^T=[a_{ji}]$ defines its adjoint operator on a Hilbert space $l^2(R)$, so that we have 
\be
\|A\|^2 =\|AA^T\|=\|A^TA\|=\rho (AA^T)=\rho (A^TA).
\label{eqT}
\ee

Given non-negative matrices $A=[a_{ij}]_{i,j\in R}$ and $B=[b_{ij}]_{i,j\in R}$,
let $A \circ B=[a_{ij}b_{ij}]_{i,j\in R}$ be the \textit{Hadamard
(or Schur) product} of $A$ and $B$
and let $A^{(t)}=[a_{ij} ^t]_{i,j\in R}$ be the \textit{Hadamard (or
Schur) power} of $A$ for $t\ge0$. Here we use the convention $0^0=1$. 

The following result was proved in \cite[Theorem 3.3]{DP05} and \cite[Theorem 5.1 and Remark 5.2]{P06} by
using only basic analytic methods and elementary facts.

\begin{theorem} Given $L \in \mathcal{L}$, let $\{A_{i j}\}_{i=1, j=1}^{k, m}$ be non-negative matrices that define 
operators on $L$.
If $\alpha _1$, $\alpha _2$,..., $\alpha _m$ are positive numbers 
such that $\sum_{j=1}^m \alpha _j \ge 1$, then the matrix 
$A:= \left(A_{1 1}^{(\alpha _1)} \circ \cdots \circ A_{1 m}^{(\alpha _m)}\right) \cdots 
\left(A_{k 1}^{(\alpha _1)} \circ \cdots \circ A_{k m}^{(\alpha _m)}\right)$
also defines an operator on $L$ 
and it satisfies the inequalities  
\begin{eqnarray}
\label{basic2}
A &\le &  
(A_{1 1} \cdots  A_{k 1})^{(\alpha _1)} \circ \cdots 
\circ (A_{1 m} \cdots A_{k m})^{(\alpha _m)} , \\
\label{norm2}
\left\|A \right\| &\le &  
\|A_{1 1} \cdots  A_{k 1}\|^{\alpha _1} \cdots \|A_{1 m} \cdots A_{k m}\|^{\alpha _m}, \\ 
\label{spectral2}
\rho \left(A \right) &\le  &
\rho \left( A_{1 1} \cdots  A_{k 1} \right)^{\alpha _1} \cdots 
\rho \left( A_{1 m} \cdots A_{k m}\right)^{\alpha _m} .
\end{eqnarray}
\label{DP}
\end{theorem}
The following special case of Theorem \ref{DP} $(k=1)$ was considered in the finite dimensional case by several authors using different methods (for references see e.g. \cite{EJS88}, \cite{DP10}, \cite{DP05}, \cite{P06}). 
\begin{corollary}
Given $L \in \mathcal{L}$, let $A_1, \ldots , A_m$ be
non-negative matrices that define operators on $L$ and
$\alpha _1$, $\alpha _2$,..., $\alpha _m$ positive numbers such that
$\sum_{i=1}^m \alpha _i \ge 1$. Then we have
\be
 \|A_1 ^{( \alpha _1)} \circ A_2 ^{(\alpha _2)} \circ \cdots \circ A_m ^{(\alpha _m)} \| \le
  \|A_1\|^{ \alpha _1}  \|A_2\|^{\alpha _2} \cdots \|A_m\|^{\alpha _m}  
\label{gl1nrm}
\ee
and
\be
 \rho(A_1 ^{( \alpha _1)} \circ A_2 ^{(\alpha _2)} \circ \cdots \circ A_m ^{(\alpha _m)} ) \le
\rho(A_1)^{ \alpha _1} \, \rho(A_2)^{\alpha _2} \cdots \rho(A_m)^{\alpha _m} .
\label{gl1vecr}
\ee
\label{dp}
\end{corollary}

The following special case of Theorem \ref{DP} was also proved in \cite[Proposition 3.1]{DP05} and \cite[Lemma 4.2]{P06}.

\begin{proposition} 
Given $L$ in $\mathcal{L}$, let $A_1, \ldots , A_m$ be non-negative matrices that define operators on $L$.
Then, for any $t \ge 1$ and $i=1, \ldots , m$, $A_i ^{(t)}$ also defines an operator on $L$, and the following
inequalities hold
\be
A_{1} ^{(t)} \cdots  A_{m} ^{(t)} \le ( A_1 \cdots A_m )^{(t)},
\label{gl}
\ee
\be
\| A_{1} ^{(t)} \cdots  A_{m} ^{(t)}\| \le \| A_1 \cdots A_m \|^{t},
\label{gl1}
\ee
\be
\rho \! \left(A_{1} ^{(t)} \cdots  A_{m} ^{(t)}\right) \le \rho (A_1 \cdots A_m)^{t}. 
\label{gl3}
\ee 
\label{2}
\end{proposition}

Note that Theorem \ref{DP} and its special cases proved to be quite useful in different contexts (see e.g. \cite{EHP90}, \cite{EJS88},  \cite{DP05}, \cite{P06},  \cite{DP10}, \cite{S11}, \cite{P12}, \cite{CZ15}).
It will also be one of the main tools in the current paper. 

Banach sequence spaces are special cases of Banach function spaces. As proved in \cite{DP05} and \cite{P06},  the inequalities in Theorem \ref{DP} and Corollary \ref{dp} can be extended to positive kernel operators on Banach function spaces provided $\sum_{i=1}^m \alpha _i = 1$.  
Since our first theorem in the next section gives an inequality for these general spaces, we shortly recall some basic definitions and results from \cite{DP05} and  \cite{P06}.

Let $\mu$ be a $\sigma$-finite positive measure on a $\sigma$-algebra $\cM$ of subsets of a non-void set $X$.
Let $M(X,\mu)$ be the vector space of all equivalence classes of (almost everywhere equal)
complex measurable functions on $X$. A Banach space $L \subseteq M(X,\mu)$ is
called a {\it Banach function space} if $f \in L$, $g \in M(X,\mu)$,
and $|g| \le |f|$ imply that $g \in L$ and $\|g\| \le \|f\|$. We will assume that  $X$ is the carrier of $L$, that is, there is no subset $Y$ of $X$ of 
 strictly positive measure with the property that $f = 0$ a.e. on $Y$ for all $f \in L$ (see \cite{Za83}).
Observe that a Banach sequence space is a Banach function space over a measure space $(R, \mu)$, 
where $\mu$ denotes the counting measure on $R$ (and for $L\in \mathcal{L}$ the set $R$ is the carrier of $L$).

As before, by an {\it operator} on a Banach function space $L$ we always mean a linear
operator on $L$.  An operator $T$ on $L$ is said to be {\it positive} 
if it maps nonnegative functions to nonnegative ones.
Given operators $S$ and $T$ on $L$, we write $S \ge T$ if the operator $S - T$ is positive.

In the special case $L= L^2(X, \mu)$ we can define the {\it numerical radius} $w(T)$ of 
a bounded operator $T$ on $L^2(X, \mu)$ by 
$$ w(T) = \sup \{ | \langle T f, f \rangle | : f \in L^2(X, \mu), \| f \|_2 = 1 \} . $$
If, in addition, $T$ is positive, then it is easy to prove that 
$$ w(T) = \sup \{ \langle T f, f \rangle  : f \in L^2(X, \mu)_+ , \| f \|_2 = 1 \} . $$
From this it follows easily that $w(S) \le w(T)$ for all positive operators $S$ and $T$ on $L^2(X, \mu)$ with $S \le T$.

An operator $K$ on a Banach function space $L$ is called a {\it kernel operator} if
there exists a $\mu \times \mu$-measurable function
$k(x,y)$ on $X \times X$ such that, for all $f \in L$ and for almost all $x \in X$,
$$ \int_X |k(x,y) f(y)| \, d\mu(y) < \infty \ \ \ {\rm and} \ \ 
   (K f)(x) = \int_X k(x,y) f(y) \, d\mu(y)  .$$
One can check that a kernel operator $K$ is positive iff 
its kernel $k$ is non-negative almost everywhere. 
For the theory of Banach function spaces we refer the reader to the book \cite{Za83}.

Let $K$ and $H$ be positive kernel operators on $L$ with kernels $k$ and $h$ respectively,
and $\alpha \ge 0$.
The \textit{Hadamard (or Schur) product} $K \circ H$ of $K$ and $H$ is the kernel operator
with kernel equal to $k(x,y)h(x,y)$ at point $(x,y) \in X \times X$ which can be defined (in general) 
only on some order ideal of $L$. Similarly, the \textit{Hadamard (or Schur) power} 
$K^{(\alpha)}$ of $K$ is the kernel operator with kernel equal to $(k(x, y))^{\alpha}$ 
at point $(x,y) \in X \times X$ which can be defined only on some ideal of $L$.

Let $K_1 ,\ldots, K_n$ be positive kernel operators on a Banach function space $L$, 
and $\alpha _1, \ldots, \alpha _n$ positive numbers such that $\sum_{j=1}^n \alpha _j = 1$.
Then the {\it  Hadamard weighted geometric mean} 
$K = K_1 ^{( \alpha _1)} \circ K_2 ^{(\alpha _2)} \circ \cdots \circ K_n ^{(\alpha _n)}$ of 
the operators $K_1 ,\ldots, K_n$ is a positive kernel operator defined 
on the whole space $L$, since $K \le \alpha _1 K_1 + \alpha _2 K_2 + \ldots + \alpha _n K_n$ by the inequality between the weighted arithmetic and geometric means. Let us recall  the following result which was proved in \cite[Theorem 2.2]{DP05} and 
\cite[Theorem 5.1]{P06}. 

\begin{theorem} 
Let $\{A_{i j}\}_{i=1, j=1}^{k, m}$ be positive kernel operators on a Banach function space $L$.
If $\alpha _1$, $\alpha _2$,..., $\alpha _m$ are positive numbers  
such that $\sum_{j=1}^m \alpha _j = 1$, then the inequalities
(\ref{basic2}), (\ref{norm2}) and (\ref{spectral2}) hold.

If,  in addition,  $L= L^2(X, \mu)$, then  

$$w \! \left( \! \left(A_{1 1}^{(\alpha _1)} \circ \cdots \circ A_{1 m}^{(\alpha _m)}\right) \ldots 
\left(A_{k 1}^{(\alpha _1)} \circ \cdots \circ A_{k m}^{(\alpha _m)} \right) \! \right)$$
\be 
\le  w \! \left( A_{1 1} \cdots  A_{k 1} \right)^{\alpha _1} \cdots 
w \! \left( A_{1 m} \cdots A_{k m}\right)^{\alpha _m}.
\label{glnum}
\ee
\label{DPBfs}
\end{theorem}
The following result is a special case  of Theorem \ref{DPBfs}. 
\begin{theorem} 
Let $A_1 ,\ldots, A_n$ be positive kernel operators on a Banach function space  $L$,
and $\alpha _1, \ldots, \alpha _n$ positive numbers such that $\sum_{j=1}^n \alpha _j = 1$.
Then the inequalities (\ref{gl1nrm}) and (\ref{gl1vecr}) hold.

If, in addition,  $L= L^2(X, \mu)$,  then
\be  
w(A_1 ^{( \alpha _1)} \circ A_2 ^{(\alpha _2)} \circ \cdots \circ A_m ^{(\alpha _m)} ) \le
w(A_1)^{ \alpha _1} \, w(A_2)^{\alpha _2} \cdots w(A_m)^{\alpha _m} .
\label{gl1num} 
\ee
\end{theorem}


\section{Results}

We begin with a new proof of (\ref{Hu}) that is based on the inequality (\ref{gl1vecr}).

\begin{theorem} 
\label{Bfs}
Let $A_1, \ldots , A_m$ be positive kernel operators on a Banach function space $L$. 
Then 
\be 
\rho \left(A_1^{\left(\frac{1}{m}\right)} \circ A_2^{\left(\frac{1}{m}\right)} \circ \cdots \circ 
A_m^{\left(\frac{1}{m}\right)}\right)   \le \rho (A_1 A_2 \cdots A_m)^{\frac{1}{m}} .  
\label{genHuBfs}
\ee

If, in addition, $L \in \mathcal{L}$ (and so $A_1, \ldots , A_m$ can be considered as non-negative matrices that define operators on $L$),  then 
\be 
\rho (A_1 \circ A_2 \circ \cdots \circ A_m) \le \rho (A_1 A_2 \cdots A_m)  .  
\label{genHu}
\ee
\end{theorem}

\begin{proof}
The block matrix 
$$ T = T(A_1, A_2, \ldots, A_m) := \left[
\begin{matrix}  
0 & A_1 & 0 &  0 & \ldots & 0 & 0 \cr
0 & 0 &  A_2 & 0  & \ldots & 0 & 0 \cr
0 & 0 & 0 & A_3  & \ldots & 0 & 0 \cr
\vdots & \vdots & \vdots & \ddots & \ddots & \vdots &\vdots \cr
\vdots & \vdots & \vdots & \vdots & \ddots & \ddots &\vdots \cr
0 & 0 & 0 & 0 & \ldots & 0 & A_{m-1} \cr 
A_m & 0 & 0 & 0 & \ldots & 0 & 0
\end{matrix}  \right]  $$
defines a positive kernel operator on the cartesian product of $m$ copies of $L$. Since $T^m$ has a diagonal form
$$ T^m = \textrm{diag} \, \left( A_1 A_2 \cdots A_m,  A_2 A_3 \cdots A_m A_1,   A_3 A_4 \cdots A_m A_1 A_2, 
\ldots,  A_m A_1 A_2 \cdots A_{m-1} \right) , $$
we have  $\rho (T)^m = \rho (T^m) = \rho ( A_1 A_2 \cdots A_m )$.

Now define $T_k := T(A_k, A_{k+1}, \ldots, A_m, A_1, \ldots, A_{k-1})$ for $k =1, 2, \ldots, m$. 
Then $\rho (T_k)^m = \rho ( A_1 A_2 \cdots A_m )$ for each $k$. Using the inequality (\ref{gl1vecr}) we obtain that
$$  \rho \left(T_1^{\left(\frac{1}{m}\right)} \circ T_2^{\left(\frac{1}{m}\right)} \circ \cdots \circ 
T_m^{\left(\frac{1}{m}\right)}\right) \le \left( \rho(T_1) \, \rho(T_2) \cdots \rho(T_m)  \right)^{\frac{1}{m}} =
\rho (A_1 A_2 \cdots A_m)^{\frac{1}{m}}  . $$  
 Since 
$$ \rho \left(T_1^{\left(\frac{1}{m}\right)} \circ T_2^{\left(\frac{1}{m}\right)} \circ \cdots \circ 
     T_m^{\left(\frac{1}{m}\right)}\right) = 
\rho \left(A_1^{\left(\frac{1}{m}\right)} \circ A_2^{\left(\frac{1}{m}\right)} \circ \cdots \circ 
A_m^{\left(\frac{1}{m}\right)}\right) , $$ 
the inequality (\ref{genHuBfs}) is proved.

If, in addition, $L \in \mathcal{L}$, then we apply the inequality 
$$  \rho (T_1 \circ T_2 \circ \cdots \circ T_m) \le \rho(T_1) \, \rho(T_2) \cdots \rho(T_m) $$
that is a special case of the inequality (\ref{gl1vecr}). We then observe that 
$\rho (T_1 \circ T_2 \circ \cdots \circ T_m) = \rho (A_1 \circ A_2 \circ \cdots \circ A_m)$
and $\rho(T_1) \, \rho(T_2) \cdots \rho(T_m) = \rho (A_1 A_2 \cdots A_m)$.
This completes the proof.
\end{proof}

It should be mentioned that the special case of inequality (\ref{genHuBfs}) for pairs of operators on $L^p$-spaces was already given in 
\cite[Theorem 2.8]{S11}.

The following theorem generalizes the inequalities (\ref{mix1}) to several matrices, and it provides an alternative proof 
of the inequality (\ref{genHu}). We also establish related inequalities for the operator norm and the numerical radius.  

\begin{theorem} 
Given $L \in \mathcal{L}$, let $A_1, \ldots , A_m$ be non-negative matrices that define operators on $L$. 
For $t \in [1,m]$ and $i=1, \ldots, m$, put \\
 $P_i = A_i^{(t)} A_{i+1}^{(t)} \cdots A_m^{(t)} A_1^{(t)} A_2^{(t)} \cdots A_{i-1}^{(t)}$.  Then 
$$\rho (A_1 \circ \cdots \circ A_m)  \le  
\rho \left( P_1^{(\frac{1}{t})} \circ \cdots \circ P_m^{(\frac{1}{t})} \right)^{ \frac{1}{m}} \le $$
\be
\le \rho (A_1^{(t)} \cdots A_m^{(t)})^{\frac{1}{t}} \le \rho ((A_1 \cdots A_m )^{(t)})^{\frac{1}{t}} 
\le \rho (A_1\cdots A_m)
\label{poslpP1}
\ee
and
$$
\|(A_1 \circ \cdots \circ A_m)^m\|  \le  \| P_1^{(\frac{1}{t})} \circ \cdots \circ P_m^{(\frac{1}{t})} \| \le
 \left( \| P_1\| \cdots \|P_m\| \right)^{ \frac{1}{t}} \le
$$
$$ \le   \left( \|  (A_1  A_2  \cdots A_m) ^{(t)}\|   \| (A_2  \cdots A_m  A_1) ^{(t)}\|  \cdots \|(A_m   A_1  \cdots  A_{m-1}) ^{(t)}\| \right)^{ \frac{1}{t}} \le $$
\be
\le   \|  A_1  A_2  \cdots A_m \|  \| A_2  \cdots A_m  A_1\| \cdots \|A_m   A_1  \cdots  A_{m-1}\| .
\label{poslpP1norm}
\ee
If, in addition, $L=l^2 (R)$ and $t= m$, then
$$
w((A_1 \circ \cdots \circ A_m)^m)  \le 
w \left( P_1^{(\frac{1}{m})} \circ \cdots \circ P_m^{(\frac{1}{m})} \right) \le 
 \left( w( P_1) \cdots w(P_m)  \right)^{ \frac{1}{m}} \le $$
\be
\le   \left( w((A_1  A_2  \cdots A_m) ^{(m)}) \, w((A_2  \cdots A_m  A_1) ^{(m)})  \cdots 
w((A_m   A_1  \cdots  A_{m-1}) ^{(m)})  \right)^{ \frac{1}{m}} .
\label{poslpP1num}
\ee
\label{genP1}
\end{theorem}
\begin{proof} 
Similarly as $P_i$, we define the Hadamard product 
$$ H_i = A_i^{(t)} \circ A_{i+1}^{(t)} \circ \cdots \circ A_m^{(t)} \circ A_1^{(t)} \circ A_2^{(t)} \circ \cdots \circ A_{i-1}^{(t)} = $$ 
$$ = (A_i \circ A_{i+1} \circ \cdots \circ A_m \circ A_1 \circ A_2 \circ \cdots \circ A_{i-1})^{(t)} = 
 (A_1 \circ \cdots \circ A_m)^{(t)} ,$$
so that, in fact, $H_1 = H_2 = \ldots = H_m$.
Let us prove the inequalities (\ref{poslpP1}). Since $\frac{m}{t} \ge 1$, we apply the inequality (\ref{basic2})
to obtain the inequality
$$ (A_1 \circ \cdots \circ A_m)^m =   H_1^{(\frac{1}{t})} \cdots H_m^{(\frac{1}{t})} \le 
     P_1^{(\frac{1}{t})} \circ \cdots \circ P_m^{(\frac{1}{t})} . $$
Therefore, we have 
$$ \rho (A_1 \circ \cdots \circ A_m)^m =\rho ((A_1 \circ \cdots \circ A_m)^m) 
\le \rho \left( P_1^{(\frac{1}{t})} \circ \cdots \circ P_m^{(\frac{1}{t})} \right) , $$
proving the first inequality in (\ref{poslpP1}). 
Since $\frac{m}{t} \ge 1$, for the proof of the second inequality in (\ref{poslpP1}) we can apply 
the inequality (\ref{gl1vecr}) to obtain that
$$ \rho \left( P_1^{(\frac{1}{t})} \circ \cdots \circ P_m^{(\frac{1}{t})} \right) \le 
\left( \rho(P_1) \cdots \rho(P_m) \right)^{\frac{1}{t}} =  \rho (A_1^{(t)} \cdots A_m^{(t)})^{\frac{m}{t}} . $$
Using the inequalities (\ref{gl}) and  (\ref{gl3}) we prove the remaining inequalities in (\ref{poslpP1}):
$$ \rho (A_1^{(t)} \cdots A_m^{(t)}) \le \rho ((A_1 \cdots A_m )^{(t)}) \le 
\rho (A_1\cdots A_m)^t . $$

The inequalities (\ref{poslpP1norm}) and (\ref{poslpP1num}) are proved in a similar way. 
\end{proof}

\begin{corollary} 
Given $L \in \mathcal{L}$, let $A$ and $B$ be non-negative matrices that define operators on $L$. 
Then, for every $t \in [1, 2]$,
$$ \rho (A \circ B)  \le  \rho \left( (A^{(t)} B^{(t)})^{(\frac{1}{t})} \circ  
(B^{(t)} A^{(t)})^{(\frac{1}{t})} \right)^{ \frac{1}{2}} \le 
\rho (A^{(t)} B^{(t)})^{\frac{1}{t}}  \le \rho ((A B)^{(t)})^{\frac{1}{t}}  \le \rho (A B ) $$
and 
$$ \| (A \circ B)^2\|  \le  \|  (A^{(t)} B^{(t)})^{(\frac{1}{t})} \circ  
(B^{(t)} A^{(t)})^{(\frac{1}{t})} \| \le 
\left( \| A^{(t)} B^{(t)}\| \| B^{(t)} A^{(t)}\| \right)^{\frac{1}{t}}  \le $$
$$ \le (\| (A B)^{(t)}\|  \| (B A)^{(t)}\| )^{\frac{1}{t}}  \le \| A B \| \| BA \| . $$
If, in addition, $L=l^2 (R)$, then
$$ w((A \circ B)^2)  \le  
w \left( (A^{(2)} B^{(2)})^{(\frac{1}{2})} \circ (B^{(2)} A^{(2)})^{(\frac{1}{2})} \right) \le  $$
$$ \le  \left( w( A^{(2)} B^{(2)}) \, w(B^{(2)} A^{(2)})  \right)^{ \frac{1}{2}} \le 
 \left( w((A  B) ^{(2)}) \, w((B  A) ^{(2)}) \right)^{ \frac{1}{2}} .$$
\end{corollary}

As a consequence of Theorem \ref{genP1} we obtain the following infinite dimensional generalization and refinement of (\ref{CZ}), which was the main result of \cite{CZ15}.
\begin{corollary} Given $L \in \mathcal{L}$ and $m\ge 2$, let $A_1, \ldots , A_m$ be
non-negative matrices that define operators on $L$. For $t \in [1,m]$ and $i=1, \ldots, m$, put 
$P_i = A_i^{(t)} A_{i+1}^{(t)} \cdots A_m^{(t)} A_1^{(t)} A_2^{(t)} \cdots A_{i-1}^{(t)}$.  Then 
$$
\rho (A_1 \circ A_2 \circ \cdots \circ A_m) \le  \rho (A_1 \circ A_2 \circ \cdots \circ A_m)^{1- \frac{t}{m}}  \rho \left( P_1^{(\frac{1}{t})} \circ \cdots \circ P_m^{(\frac{1}{t})} \right)^{ \frac{t}{m^2}}
$$
$$\le \rho (A_1 \circ A_2 \circ \cdots \circ A_m)^{1- \frac{t}{m}} \rho (A_1^{(t)} \cdots A_m^{(t)} )^{\frac{1}{m}} $$
\be
\le   \rho (A_1 \circ A_2 \circ \cdots \circ A_m)^{1- \frac{t}{m}}  \rho ((A_1 \cdots A_m )^{(t)}) ^{\frac{1}{m}} \le \rho (A_1 A_2 \cdots A_m).
\label{refCZ}
\ee
\label{fromCZ}
\end{corollary}
\begin{proof} Since 
 $$\rho (A_1 \circ A_2 \circ \cdots \circ A_m) =  \rho (A_1 \circ A_2 \circ \cdots \circ A_m)^{1- \frac{t}{m}} \rho (A_1 \circ A_2 \circ \cdots \circ A_m)^{\frac{t}{m}},$$
the result follows by applying (\ref{poslpP1}). 
\end{proof}

By applying Theorems  \ref{DP} and  \ref{genP1} we obtain
the following result which generalizes  \cite[Proposition 2.4]{CZ15} and generalizes and refines  \cite[Theorem 4]{Hu11}.
\begin{corollary}
Let $A_1, \ldots , A_m$ be
non-negative matrices that define operators on $l^2(R)$ and $t \in [1,m]$. 
If we denote $S_i = A_i A^T _i$  and  \\
$T_i = S_i^{(t)} S_{i+1}^{(t)} \cdots S_m^{(t)} S_1^{(t)} S_2^{(t)} \cdots S_{i-1}^{(t)}$ for $i=1, \ldots , m$, then 
$$
 \|A_1 \circ A_2 \circ \cdots \circ A_m \|^2 \le  \rho (S_1 \circ S_2 \circ \cdots \circ S_m)  \le  
\rho \left( T_1^{(\frac{1}{t})} \circ \cdots \circ T_m^{(\frac{1}{t})} \right)^{ \frac{1}{m}}
$$
\be
\le \rho (S_1^{(t)} \cdots S_m^{(t)})^{\frac{1}{t}} \le \rho ((S_1 \cdots S_m )^{(t)})^{\frac{1}{t}} 
\le \rho (S_1\cdots S_m) . 
\label{th4nr}
\ee
\label{th4Hu}
\end{corollary}

\begin{proof} 
By Theorem \ref{DP} we have 
$$(A_1 \circ A_2 \circ \cdots \circ A_m)(A_1 \circ A_2 \circ \cdots \circ A_m)^T =(A_1 \circ A_2 \circ \cdots \circ A_m)(A_1 ^T\circ A_2 ^T \circ \cdots \circ A_m ^T)  $$
$$\le (A_1 A_1 ^T)\circ( A_2 A_2 ^T)\circ \cdots \circ (A_m A_m ^T)= S_1 \circ S_2 \circ \cdots \circ S_m$$
and so it follows by (\ref{eqT}) and Theorem \ref{DP}
$$ \|A_1 \circ A_2 \circ \cdots \circ A_m \|^2 =\rho ( (A_1 \circ A_2 \circ \cdots \circ A_m)(A_1 \circ A_2 \circ \cdots \circ A_m)^T) \le \rho (S_1 \circ S_2 \circ \cdots \circ S_m),  $$
which proves the first inequality (\ref{th4nr}). Now the result follows by applying (\ref{poslpP1}).
\end{proof}

The following Cauchy-Schwarz type inequality for the spectral radius of $n \times n$ non-negative matrices  was proved in \cite[Proposition 2.6]{CZ15} using the trace description: if $A$, $B$ are $n \times n$  non-negative matrices, then
\be
\rho (A\circ B ) \le \rho (A \circ A)^{1/2}\rho (B \circ B)^{1/2}.
\label{CSsp}
\ee
This result has already been implicitly known and also applied (see e.g. the proof of \cite[Theorem 3.7]{P12}). Moreover, an easy application of Corollary \ref{dp} gives the following infinite-dimensional generalization of (\ref{CSsp}) and its analogues for the operator norm and the numerical radius.

\begin{theorem} Given $L \in \mathcal{L}$, let $A_1, \ldots , A_m$ be
non-negative matrices that define operators on $L$. 
Define functions $r, N : [1, \infty)  \mapsto \RR$ by
$$ r(t) =  \left( \rho ( A_1 ^{(t)}) \rho ( A_2 ^{(t)}) \cdots  \rho ( A_m ^{(t)}) \right)^{1/t}  \ \ \ \textrm{and}  \ \ 
   N(t) = \left( \| A_1 ^{(t)}\| \|A_2 ^{(t)}\| \cdots  \| A_m ^{(t)}\| \right)^{1/t}  . $$
Then the function $r$ is  decreasing on $ [1, \infty) $, and $\rho (A_1 \circ A_2 \circ \cdots \circ A_m)$ is its lower bound on the interval $[1, m]$.
Similarly,  the function $N$ is  decreasing on $ [1, \infty)$, and $\|A_1\circ A_2 \circ \cdots \circ A_m  \|$ is its lower bound on the interval $[1,m]$.

If, in addition,  $L = l^2 (R)$ then
\be
w (A_1\circ A_2 \circ \cdots \circ A_m  )  \le \left(w ( A_1 ^{(m)})  w ( A_2 ^{(m)}) \cdots  w ( A_m ^{(m)})\right)^{1/m}.
\label{genCSnum}
\ee
\label{decreasing}
\end{theorem}

\begin{proof} 
The  expression $ \rho ( A_i ^{(t)})^{1/t}$ is decreasing in $t \in [1,\infty)$. Indeed, if $s \ge t >0$ then 
the inequality (\ref{gl3}) implies that 
$$ \rho \left(A_i ^{(s)}\right)^{1/s}=\rho \left(\left(A_i ^{(t)}\right)^{\left(\frac{s}{t}\right)}\right)^{1/s} \le 
\rho \left(A_i ^{(t)}\right)^{1/t} . $$
So, it follows that the function $r$ is  decreasing.

If $1\le t \le m$, then  $\frac{m}{t} \ge 1$, and so we have by  (\ref{gl1vecr})
$$ r(t) \ge \rho ((A_1  ^{(t)} )^{(1/t)} \circ (A_2^{(t)} )^{(1/t)} \circ \cdots \circ (A_m ^{(t)} )^{(1/t)} ) = 
\rho (A_1\circ A_2 \circ \cdots \circ A_m)  . $$
Therefore, on the interval $[1,m]$ the function $r$ is bounded below by $\rho (A_1 \circ A_2 \circ \cdots \circ A_m)$.

In a similar manner one can show the properties of the function $N$.
Furthermore, the inequality (\ref{genCSnum}) follows from the inequality (\ref{gl1num}).
\end{proof}
\begin{remark}{\rm In the case when  $L=\CC ^n$ and $A_1, \ldots , A_m$ are $n \times n$ non-negative matrices, then the functions $t \mapsto r(t)$ and $t \mapsto N(t)$ from  Theorem \ref{decreasing} are well-defined decreasing functions on $(0, \infty)$, with lower bounds on the interval $(0, m]$ equal to  $\rho (A_1 \circ A_2 \circ \cdots \circ A_m)$ and $\|A_1 \circ A_2 \circ \cdots \circ A_m\|$, respectively.

Indeed, this follows from the proof of Theorem \ref{decreasing} by replacing the intervals $[1, \infty)$ and $[1,m]$ with $(0, \infty)$ and $(0,m]$, respectively. 
}
\end{remark}

\begin{remark}
{\rm In general, we do not have that $\rho (A_1\circ A_2 \circ \cdots A_m  ) \le r(t)$ for $t > m$.
For example, in the case $m=1$ take $A_1=\left[\begin{array}{cc} 1 & 1 \\1 & 1 \\ 
\end{array}\right]$. Then  $\rho(A_1)=2 > \rho (A_1 ^{(t)})^{1/t}=2^{1/t} $ for $t > 1$.
This matrix can be also used  in the general case  $m \ge 2$. 
Setting $A_k :=A_1$  for $k =2, \ldots, m$ we have
 $\rho (A_1 \circ A_2 \circ \cdots \circ A_m) = \rho(A_1)= 2 >  
\left( \rho ( A_1 ^{(t)}) \rho ( A_2 ^{t)}) \cdots  \rho ( A_m ^{(t)}) \right)^{1/t} = 2^{m/t}$ 
for $t > m$.

Note that the limit $\mu (A):= \lim _{k\to \infty}  \rho ( A ^{(t)})^{1/t}$ plays (at least in the case of $n \times n$ non-negative matrices) the role of the spectral radius in the algebraic system max algebra (see e.g. \cite{B98},  \cite{P08}, \cite{EJS88}, \cite{EHP90}, \cite{F86}, \cite{ED08}, \cite{Bu10}, \cite{GA15} and the references cited there for various applications). 
\label{decr}
}
\end{remark}
\begin{remark}{\rm We can use an example from \cite{DP05} to show that the product
$$ \left( w (A_1^{(t)}) w (A_2^{(t)}) \cdots  w (A_m^{(t)}) \right)^{1/t} $$
is not necessarily decreasing in $t$. Let $L=\CC ^2$ and
$$A =\left[
\begin{matrix}
0 & 1 \\
0 & 0 \\
\end{matrix}\right].$$
Then $A^{(t)}=A$ for all $t > 0$, $w(A)=\frac{1}{2}$, and so $w(A^{(t)})=\frac{1}{2} > \left(\frac{1}{2}\right)^t=w(A)^t$ for $t >1$.  Therefore, choose $A_1 = \ldots = A_m = A$ above.
\label{num_counter}
}
\end{remark}

The following result generalizes \cite[Theorem 5]{Hu11}.
\begin{theorem} 
Let $A_1, \ldots , A_m$ be
non-negative matrices that define operators on $l^2(R)$.
If $m$ is even, then
$$ \|A_1 \circ A_2 \circ \cdots \circ A_m \|^2 \le 
  \rho (A_1 ^TA_2A_3 ^TA_4 \cdots A_{m-1} ^T A_m   ) \rho (A_1A_2 ^TA_3A_4 ^T \cdots A_{m-1} A_m ^T  )$$
\be
= \rho (A_1 ^TA_2A_3 ^TA_4 \cdots A_{m-1} ^T A_m   )\rho (A_mA_{m-1} ^T \cdots A_{4} A_3 ^T A_2A_1 ^T ).
\label{th5sod_notref}
\ee
If $m$ is odd, then
\be
 \|A_1 \circ A_2 \circ \cdots \circ A_m \|^2 \le  \rho (A_1 A_2 ^TA_3A_4 ^T \cdots A_{m-2} A_{m-1} ^T A_m A_1 ^T A_2 A_3 ^TA_4 \cdots A_{m-2}^T A_{m-1} A_m ^T  )
\label{th5lih_notref}
\ee
\label{th5H}
\end{theorem}
\begin{proof} If $m$ is even, we have
by (\ref{basic2})
$$((A_1 \circ A_2 \circ \cdots \circ A_m)^T(A_1 \circ A_2 \circ \cdots \circ A_m))^\frac{m}{2}$$
$$=(A_1 ^T \circ A_2 ^T \circ \cdots \circ A_m ^T)(A_2 \circ  \cdots \circ A_m \circ A_1)(A_3 ^T \circ A_4 ^T \circ \cdots \circ A_m ^T\circ A_1 ^T \circ A_2^T )$$
$$(A_4 \circ \cdots \circ A_m\circ A_1 \circ A_2 \circ A_3) \cdots (A_{m-1 } ^T\circ A_m ^T \circ A_{1} ^T \circ \cdots \circ A_{m-2}^T)(A_m \circ A_1 \circ \cdots \circ A_{m-1})$$
$$\le (A_1 ^T A_2 A_3 ^T A_4\cdots A_{m-1} ^T A_m)\circ (A_2 ^T A_3 A_4 ^T A_5\cdots A_{m} ^T A_1)\circ \cdots $$
$$\circ (A_{m-1} ^T A_m A_{1} ^T A_2 \cdots A_{m-3} ^T A_{m-2} )\circ(A_{m} ^T A_1A_{2} ^T A_3 \cdots A_{m-2} ^T A_{m-1} ) $$
It follows by (\ref{spectral2}) that
$$\|A_1 \circ A_2 \circ \cdots \circ A_m \|^{m}= \rho ((A_1 \circ A_2 \circ \cdots \circ A_m)^T(A_1 \circ A_2 \circ \cdots \circ A_m))^\frac{m}{2}$$
\begin{equation}
\le \rho  ((A_1 ^T A_2 A_3 ^T A_4\cdots A_{m-1} ^T A_m)\circ (A_2 ^T A_3 A_4 ^T A_5\cdots A_{m} ^T A_1)\circ \cdots 
\label{th5sod}
\end{equation}
$$\circ (A_{m-1} ^T A_m A_{1} ^T A_2 \cdots A_{m-3} ^T A_{m-2} )\circ(A_{m} ^T A_1A_{2} ^T A_3 \cdots A_{m-2} ^T A_{m-1} )) $$
$$\le \rho (A_1 ^T A_2 A_3 ^T A_4\cdots A_{m-1} ^T A_m) \rho(A_2 ^T A_3 A_4 ^T A_5\cdots A_{m} ^T A_1)\cdots$$
$$\cdots \rho(A_{m-1} ^T A_m A_{1} ^T A_2 \cdots A_{m-3} ^T A_{m-2} )\rho(A_{m} ^T A_1A_{2} ^T A_3 \cdots A_{m-2} ^T A_{m-1} ) $$
$$=\rho ^{\frac{m}{2}} (A_1 ^TA_2A_3 ^TA_4 \cdots A_{m-1} ^T A_m   )\rho ^{\frac{m}{2}} (A_1A_2 ^TA_3A_4 ^T \cdots A_{m-1} A_m ^T  ),$$
which proves (\ref{th5sod_notref}).

If $m$ is odd, we have by (\ref{basic2})
$$((A_1 \circ A_2 \circ \cdots \circ A_m)^T(A_1 \circ A_2 \circ \cdots \circ A_m))^m$$
$$=(A_1 ^T \circ A_2 ^T \circ \cdots \circ A_m ^T)(A_2 \circ  \cdots \circ A_m \circ A_1)(A_3 ^T \circ A_4 ^T \circ \cdots \circ A_m ^T\circ A_1 ^T \circ A_2^T )$$
$$(A_4 \circ \cdots \circ A_m\circ A_1 \circ A_2 \circ A_3) \cdots (A_{m-1 }\circ A_m \circ A_{1} \circ \cdots \circ A_{m-2})(A_m ^T \circ A_1 ^T\circ \cdots \circ A_{m-1}^T)$$
$$(A_1  \circ A_2  \circ \cdots \circ A_m )(A_2 ^T \circ  \cdots \circ A_m^T \circ A_1^T)(A_3 \circ A_4  \circ \cdots \circ A_m \circ A_1  \circ A_2 )\cdots$$
$$ \cdots (A_{m-1 }^T\circ A_m^T \circ A_{1}^T \circ \cdots \circ A_{m-2}^T)(A_m \circ A_1 \circ \cdots \circ A_{m-1}) \le $$
$$(A_1 ^T A_2 A_3^TA_4  \cdots A_{m-1} A_m ^TA_1  A_2^T A_3 A_4^T \cdots A_{m-1}^T A_m) \circ (A_2 ^T A_3A_4^T  \cdots A_{m-1}^T A_m A_1^T  A_2 A_3^T A_4 $$
$$\cdots A_{m-1} A_m ^T A_1) \circ \cdots \circ (A_m ^TA_1  A_2^T A_3 A_4^T \cdots A_{m-1}^T A_mA_1 ^T A_2 A_3^TA_4  \cdots A_{m-1}). $$
It follows by (\ref{spectral2}) that
\begin{equation}
\|A_1 \circ A_2 \circ \cdots \circ A_m \|^{2m} 
\label{th5lih}
\end{equation}
$$ \le \rho ((A_1 ^T A_2 A_3^TA_4  \cdots A_{m-1} A_m ^TA_1  A_2^T A_3 A_4^T \cdots A_{m-1}^T A_m) \circ $$ 
$$ \circ (A_2 ^T A_3A_4^T  \cdots A_{m-1}^T A_m A_1^T  A_2 A_3^T A_4 \cdots A_{m-1} A_m ^T A_1) \circ \cdots $$
$$ \cdots \circ (A_m ^TA_1  A_2^T A_3 A_4^T \cdots A_{m-1}^T A_mA_1 ^T A_2 A_3^TA_4  \cdots A_{m-1})) $$
$$\le  \rho ^{\frac{m+1}{2} }(A_1 ^T A_2 A_3^TA_4  \cdots A_{m-1} A_m ^TA_1  A_2^T A_3 A_4^T \cdots A_{m-1}^T A_m)\times$$
$$\rho ^{\frac{m-1}{2} } (A_1 A_2 ^TA_3A_4 ^T \cdots A_{m-1} ^T A_mA_1 ^T A_2 A_3 ^TA_4 \cdots A_{m-1} A_m ^T  ) $$
$$= \rho ^m (A_1 A_2 ^TA_3A_4 ^T \cdots A_{m-1} ^T A_mA_1 ^T A_2 A_3 ^TA_4 \cdots A_{m-1} A_m ^T  ), $$
which completes the proof.
\end{proof}
The following result  follows from Theorem \ref{th5H} and its proof. It generalizes and refines \cite[Corollary 6]{Hu11} and \cite[Corollary 2.3]{CZ15}.
\begin{corollary}  
Let $A$,$B$ and $C$ be
non-negative matrices that define operators on $l^2(R)$. Then 
\be
\|A\circ B\| \le \rho  ^{\frac{1}{2}} ((A^T B )\circ (B ^T A)) \le \rho  (A ^TB) 
\label{atb}
\ee
and
\be
\|A\circ B \circ C\|  \le \rho^{\frac{1}{6}} ((A ^T B C^TA  B^T C) \circ (B ^T CA^T  BC^T A) \circ  (C ^TA B^T C A^TB)) 
\label{abtc}
\ee
$$ \le  \rho ^{\frac{1}{2}} (A B ^TCA ^T B C^T).$$
\end{corollary}
\begin{proof} It follows by (\ref{th5sod}) that
$$\|A\circ B\| \le \rho  ^{\frac{1}{2}} ((A^T B )\circ (B ^T A)) \le  \rho  ^{\frac{1}{2}} (A^T B )   \rho  ^{\frac{1}{2}}(B ^T A)= \rho  (A ^TB), $$
which proves (\ref{atb}).

Similarly (\ref{abtc}) follows from (\ref{th5lih}).
\end{proof}
The inequalities (\ref{abtc}) yield the following lower bounds for the operator norm of the Jordan triple product $ABA$.
\begin{corollary}
Let $A$ and $B$ be
non-negative matrices that define operators on $l^2(R)$. Then 
\be
\|A\circ B^T  \circ A\|  \le \rho^{\frac{1}{6}} ((A ^T B^T A^TA  B A) \circ (B  AA^T  B^TA^T A) \circ  (A ^TA B A A^TB^T)) \le 
\|ABA\|
\label{Jordan}
\ee
\end{corollary}
\begin{proof} It follows by (\ref{abtc}) that
$$\|A\circ B^T \circ A\|\le \rho^{\frac{1}{6}} ((A ^T B^T A^TA  B A) \circ (B  AA^T  B^TA^T A) \circ  (A ^TA B A A^TB^T))$$
$$  \le  \rho ^{\frac{1}{2}} (A B AA ^T B^T A^T) =  \|ABA\|, $$
which completes the proof.
\end{proof}
In contrast to (\ref{Jordan}) the inequality $\|A\circ B \circ A\|\le  \|ABA\|$ is not valid in general as the following example from  \cite{Hu11} shows.
\begin{example}{\rm If  $A =\left[
\begin{matrix}
0 & 1 \\
0 & 1 \\
\end{matrix}\right]$ and $B =\left[
\begin{matrix}
1 & 1 \\
0 & 0 \\
\end{matrix}\right]$, then $\|A\circ B \circ A\|=1 >0=  \|ABA\|$.
}
\label{counter}
\end{example}

Note that the inequalities (\ref{atb})  refine  the well-known inequality  $\|A\circ B\|\le \|A\|\|B\|$ and that we have 
$$\rho(A\circ B) \le\|A\circ B\| \le  \rho  ^{\frac{1}{2}} ((A^T B )\circ (B ^T A)) \le \rho (A^TB) \le \|A^TB\|\le \|A\|\|B\|.$$
Note also that $\|A\circ B \| \le \rho (AB)$ is not valid in general  as the matrices from Example \ref{counter} show (as it has already been pointed out in \cite{Hu11}).

\vspace{3mm}

We conclude the paper by combining the spectral mapping theorem for analytic functions and the inequality (\ref{genHu}). 
To this end, let $\AAA _+$ denote the collection of all power series 
$$ f(z)=\sum_{j=0}^\infty\al_j z^j $$ 
having nonnegative coeficients $\al_j \ge 0$ ($j=0,1,\dots$). 
Let $R_f$ be the radius of convergence of $f\in\AAA _+$, that is, we have
$$ \frac{1}{R_f}= \limsup_{j\to\infty} \al_j^{1/j} . $$
If $A$ is an operator on a Banach space such that $\rho (A)<R_f$, then the operator $f(A)$ is defined by 
$$ f(A)=\sum_{j=0}^\infty\al_j A^j . $$
 
\begin{theorem}
 Given $L \in \mathcal{L}$,  let  $A_1, \ldots, A_m$ be
non-negative matrices that define operators on $L$. 
If $f\in\AAA _+$ 
and $\rho(A_1   \cdots  A_m) < R_f$, then 
$$ \rho (f(A _1  \circ \cdots \circ A _m))\le \rho (f(A_1  \cdots  A_m)).$$ 
\label{powerineq}
\end{theorem}
\begin{proof}
 If $\rho(A_1  \cdots  A_m) < R_f$, then  it follows from the spectral mapping theorem and  (\ref{genHu}) that
$$\rho (f(A _1  \circ \cdots \circ A _m))= f (\rho (A _1  \circ \cdots \circ A _m)) $$
$$\le f( \rho(A_1  \cdots A_m))= \rho (f (A_1  \cdots  A_m)),$$
which completes the proof.
\end{proof}

Choosing the exponential series and the C. Neumann series for $f\in\AAA _+$, we obtain the following corollaries.

\begin{corollary}
 Given $L \in \mathcal{L}$,  let  $A_1, \ldots, A_m$ be
non-negative matrices that define operators on $L$. 
Then 
$$ \rho (\exp (A _1  \circ \cdots \circ A _m))\le \rho (\exp (A_1  \cdots  A_m)).$$ 
\label{exp}
\end{corollary}
\begin{corollary}
 Given $L \in \mathcal{L}$,  let  $A_1, \ldots, A_m$ be
non-negative matrices that define operators on $L$. 
If  $\lambda > \rho (A_1  \cdots  A_m)$, then 
$$ \rho ( (\lambda I - A _1  \circ \cdots \circ A _m)^{-1})\le \rho ( (\lambda I - A _1  \cdots A _m)^{-1}).$$ 
\label{resolv}
\end{corollary}
\noindent {\bf Acknowledgement.}  This work was supported in part by grant P1-0222 of the Slovenian Research Agency.

\bibliographystyle{amsplain}

\end{document}